\documentclass[12pt,a4paper]{amsart} 
\usepackage{latexsym}
\usepackage{amsmath}
\usepackage{amsthm}
\usepackage{tikz}
\usepackage{latexsym}
\usepackage{amssymb}
\usepackage{mdframed}
\usepackage{hyperref}

\renewcommand{\phi}{\varphi}
\renewcommand{\emptyset}{\varnothing}
\renewcommand{\setminus}{-}

\newcommand{\CAT}{{\rm CAT$(0)$}}

\makeatletter
\def\theenumi{\@roman\c@enumi}
\makeatother

\theoremstyle{plain}
\newtheorem*{NewPropositionA}{Proposition A}	
\newtheorem*{NewLemmaB}{Lemma B}	

\newtheorem*{NewTheoremC}{Theorem C}

\newtheorem{theorem}[subsection]{Theorem}

\newtheorem{lemma}[subsection]{Lemma}

\newtheorem{example}[subsection]{Example}
\newtheorem{proposition}[subsection]{Proposition}
\theoremstyle{definition}
\newtheorem{definition}[subsection]{Definition}

\newtheorem{remark}[subsection]{Remark}

\newtheorem*{theo}{Lemma 9.5 in Moussong's work}
\newenvironment{ftheo}
  {\begin{mdframed}\begin{theo}}
  {\end{theo}\end{mdframed}}

\newtheorem*{theo1}{Lemma 9.7 in Moussong's work}
\newenvironment{ftheo1}
  {\begin{mdframed}\begin{theo1}}
  {\end{theo1}\end{mdframed}}
  \newtheorem*{theo2}{Lemma 9.11 in Moussong's work}
\newenvironment{ftheo2}
  {\begin{mdframed}\begin{theo2}}
  {\end{theo2}\end{mdframed}}

\title[]{A note on almost negative matrices and Gromov-hyperbolic Coxeter groups}
\author{Philip M\"oller}
\date{\today}

\address{Philip M\"oller\\
	Department of Mathematics\\
	University of M\"unster\\ 
	Einsteinstra\ss e 62\\
	48149 M\"unster (Germany)}
\email{philip.moeller@uni-muenster.de}

\begin{document}
\begin{abstract} 
In this article we revisit Moussong's Characterization of Gromov-hyperbolic Coxeter groups. Moussong showed that a Coxeter group is Gromov-hyperbolic if and only if it does not contain a subgroup isomorphic to $\mathbb{Z}^2$. This property can be read off directly from the defining graph. We show that there is a gap in the original argument and provide a workaround.

\vspace{0.5cm}
\hspace{-0.6cm}
\medskip

\medskip
\hspace{-0.5cm}{\bf 2010 Mathematics Subject Classification.} Primary: 20F55 ; Secondary: 51F15.
\end{abstract}

\thanks{This work is funded by the Studienstiftung des deutschen Volkes and by the Deutsche Forschungsgemeinschaft (DFG, German Research Foundation) under Germany's Excellence Strategy EXC 2044--390685587, Mathematics M\"unster: Dynamics-Geometry-Structure. This work is part of the PhD project of the author.}

\maketitle
\section{Introduction}
The class of Coxeter groups has been studied since their introduction in 1935 by Coxeter \cite{Coxeter}. For their definition we need a symmetric $n\times n$ matrix $M$ with entries in $\mathbb{N}_{>0}\cup\{\infty\}$ such that $m_{ii}=1$ and $m_{ij}\geq 2$ or $m_{ij}=\infty$ for $i\neq j\in\{1,...,n\}$. Given such a matrix, the corresponding Coxeter system $(W,S)$ is given by $S=\{s_1,s_2,...,s_n\}$ and
$$ W=\langle S| (s_{i}s_{j})^{m_{ij}}=1\quad \text{for all }i,j\in {1,...,n}\text{ such that } m_{ij}\neq \infty \rangle$$
In 1935 Coxeter was able to classify the finite and affine Coxeter groups in \cite{Coxeter}. More work e.g. about the word problem was done in the 60s and 70s by J. Tits, e.g. in \cite{Tits}. Coxeter groups turned out to be very `geometric' objects, which are for example connected to the theory of buildings. Coxeter groups have been studied since their introduction from a geometric and an algebraic point of view.

In the 1980s, Gromov basically invented the branch of mathematics called \textit{geometric} group theory. The idea is to study groups via their actions on spaces with nice properties. Two very important types of spaces are the so called \CAT\ spaces and the $\delta$-hyperbolic spaces. A geodesic metric space is called \CAT\ if geodesic triangles are not thicker than Euclidean comparison triangles. A geodesic metric space is called $\delta$-hyperbolic for a real number $\delta\geq 0$ if for every geodesic triangle, a $\delta$-neighborhood of two sides contains the third. For the remainder of this article we will call a $\delta$-hyperbolic space Gromov-hyperbolic since the actual value of $\delta$ plays no role. Two standard references for geometric group theory are Gromov's original work \cite{Gromov} and the book by Bridson and Haefliger \cite{BH}.

A group is called \CAT \  (resp. Gromov-hyperbolic) if it acts geometrically, that is properly discontinuously and cocompactly via isometries on a \CAT \ (resp. Gromov-hyperbolic) space (see e.g. \cite{Gromov}). Being \CAT\ or Gromov-hyperbolic has strong implications for the structure of the group, see e.g. \cite[Thm I.4.1]{Davis}. By now we also know many examples (and non-examples) of \CAT\ and Gromov-hyperbolic groups. Some examples of CAT$(0)$ groups are Coxeter groups, right-angled Artin groups, Graph products of finite vertex groups and many more. Examples of Gromov-hyperbolic groups are free groups, Coxeter groups without a subgroup isomorphic to $\mathbb{Z}^2$ and more. Additional examples and information on these groups can be found in books on geometric group theory such as \cite{DK}.

During the time of Moussong's thesis however, not so many examples of these groups were known, and Coxeter groups were natural candidates to study with regard to these properties, since they are ``reflection'' groups. It turns out that in fact all Coxeter groups have the \CAT\ property by \cite[Thm. A]{Moussong}. The most important tool in Moussong's thesis is ``Moussong's Lemma'' \cite[Lemma II.5.21]{BH}, a technical result about metric properties of spherical complexes with ``large'' edge lengths. 

 The Gromov-hyperbolic case is a bit more difficult, since a Gromov-hyperbolic group cannot contain a subgroup isomorphic to $\mathbb{Z}^2$, some Coxeter groups cannot have this property. However for the Coxeter group case, this is the only restriction, i.e. a Coxeter group is Gromov-hyperbolic if and only if it does not contain a subgroup isomorphic to $\mathbb{Z}^2$ \cite[Thm. B]{Moussong}. This can be read off the defining graph or matrix easily. However it turns out that there is a gap in the proof of \cite[Thm. B]{Moussong}, which we will discuss here. More precisely Lemmas 9.5, 9.7 and 9.11 in \cite{Moussong} are not quite correct and the latter ones depends on the former ones. Lemma 9.11 plays a major role in the proof of Theorem B.
 
 We want to point out that the entire CAT$(0)$ case is problem-free, this includes \cite[Thm. A]{Moussong} and Moussong's Lemma.
 
 To circumvent the problems caused by these lemmas we will have to investigate almost negative matrices and their nerve complexes which were very important tools in Moussong's argument. A symmetric $n\times n$-matrix is called almost negative if all entries are real numbers and all off-diagonal entries are non-positive. The most important example is the cosine matrix of a Coxeter system, that is, the $n\times n$ matrix whose $(i,j)$ entry is $-\cos(\pi/m_{ij})$ if $m_{ij}<\infty$ and $-1$ else. We will briefly revisit the original argument in Section 3 and explain why these matrices play such a key role. Very importantly, one can define a \textit{link}-matrix, which gives an algebraic construction closely related to the geometric link in the nerve complex. More precisely the link matrix is obtained by the following construction: We start with an almost negative $n\times n$ matrix $A$ and a subset $I\subseteq\{1,...,n\}$ such that the matrix $A_I:=(a_{ij})_{i,j\in I}$ is positive definite. Further let $e_i$ denote the $i$-th standard basis vector and let $\phi\colon \mathbb{R}^n\to U^\perp:=\left\{v\in \mathbb{R}^n| v^TAe_i=0\ \forall \ i\in I\right\}$ denote the orthogonal projection with regard to the symmetric bilinear form induced by $A$ on $\mathbb{R}^n$. Note that this is well-defined, because of a slight modification of the Gram-Schmidt procedure, where one subtracts a linear combination of vectors in $U$. The link-matrix of $A$ with respect to $I$ is the $J\times J$ matrix $lk(I,A)=(b_{st})_{st \in J}$ where $J=\{1,...,n\}-I$ and $b_{st}=\phi(e_s)^TA\phi(e_t)$.
 
 We prove the following proposition to locate the precise problem in the original argument.
 \begin{NewPropositionA}
 	Lemmas 9.5, 9.7 and 9.11 in \cite{Moussong} are correct if the almost negative matrix in question has no link-matrix which has a principal submatrix of order $\geq 2$ which is parabolic or has a row or column consisting of zeroes. There exist counterexamples to the general statements.
 \end{NewPropositionA}
 In these problematic cases we prove the following technical lemma
 \begin{NewLemmaB}
 	Let $A$ denote an almost negative matrix such that one of its link matrices has a row or column consisting entirely of zeroes. Then $A$ splits as a direct sum $A=B\oplus C$.
 \end{NewLemmaB}
 This will be the main ingredient to fill the gap in the proof of \cite[Thm. B]{Moussong}. To do so we provide an alternative proof of \cite[Lemma 10.3]{Moussong} in the potentially problematic cases, which turns out to be sufficient to complete the proof. 
 \begin{NewTheoremC}
 	Let $(W,S)$ denote a Coxeter system. The following are equivalent:
 	\begin{enumerate}
 		\item $W$ is Gromov-hyperbolic.
 		\item $W$ contains no subgroup isomorphic to $\mathbb{Z}^2$.
 		\item There is no subset $T\subseteq S$ such that $(W_T,T)$ is an affine Coxeter system of rank $\geq 3$ and there is no pair of disjoint subsets $T_1,T_2\subseteq S$ such that $W_{T_1}$ and $W_{T_2}$ are infinite and commute.
 	\end{enumerate}
 \end{NewTheoremC}
To our knowledge, there is no proof of this result in the literature which does not depend on the incorrect lemmas of Moussong's work. For example in \cite{Davis} the dependence comes from Lemma I.7.6, which is precisely \cite[Lemma 10.3]{Moussong}.

The paper is organized as follows: In Section 2 we will review the basic definitions, in Section 3 we will recall the structure of the original argument and prove Proposition A and in Section 4 we will prove Lemma B and Theorem C.\\

{\bf Acknowledgement.} The author thanks Olga Varghese and Linus Kramer for many useful comments and remarks about an earlier version of this manuscript, as well as for their encouragement to finish this project. The author also wants to thank the anonymous referee for their detailed report with many useful comments and remarks. Finally the author wants to express his gratitude to Lara Bessmann for a multitude of useful comments on an earlier version of this article.
\section{Almost negative matrices, nerve-complexes and the Davis complex}
As was already stated in the introduction, an \textit{almost negative matrix} is a symmetric $n\times n$ matrix with real entries such that all off-diagonal entries are non-positive. We say that an $n\times n$-matrix has \textit{order} $n$.
We now define a complex using an almost negative matrix, which will be a spherical cell complex, an $M_1$ complex in the language of \cite{BH}. Let $A=(a_{ij})_{i,j\in\{1,...,n\}}$ denote an almost negative $n\times n$ matrix and $I\subseteq\{1,...,n\}$ such that $A_I$ is positive definite. Let $B_I$ denote the spherical simplex in $\mathbb{S}^{|I|+1}$ defined by the linear span of $\{e_i\}_{i\in I}$ with respect to the positive definite form induced by $A$. That means $B_I$ is the simplex with vertices $v_i=\frac{e_i}{\sqrt{e_i^{t}Ae_i}}$ for $i\in I$ and $B_I$ carries a natural metric induced by the positive definite form induced by $A$. The \textit{nerve complex} $N(A)$ is then defined as the cell complex obtained by the union of all cells $B_I$, where two simplices $B_I$ and $B_J$ are glued together along $B_{I\cap J}$. In the context of cell complexes, links are very important. In the context of almost negative matrices and nerve complexes, it is possible to `compute the link algebraically' using the link matrix defined in the introduction. More precisely we have:
\begin{lemma}(\cite[p. 21]{Moussong})\label{1.3}
	Let $A$ be an almost negative $n\times n$ matrix and $I\subseteq \{1,...,n\}$, then the link of the simplex $B_I$ in $N(A)$ (see \cite[I.7.14]{BH} for the precise definition), $lk(B_I,N(A))$, is isometric to the nerve of the link matrix $N(lk(I,A))$. Furthermore we have $lk(I,A)=lk\left(\{i\},lk\left(I\setminus\{i\},A\right)\right)$ for all $i\in I$. Moreover $lk(I,A)$ is also an almost negative matrix.
\end{lemma}
The main application for these matrices and links is the \textit{Davis complex with the Moussong metric} for Coxeter groups. Let us briefly recall the definition of this fundamental object of study. We first discuss the \textit{Davis complex}, for more material see \cite[Chapters 7 and 12]{Davis}. Let $(W,S)$ denote a Coxeter system and let $\mathcal{S}$ denote the set of all \textit{spherical} subsets of $S$, that is subsets $T\subseteq S$ such that $\langle T\rangle=W_T $ is finite. Let $\mathcal{WS}:=\{wW_T\mid w\in W, T\in \mathcal{S} \}$. This is a partially ordered set with respect to $``\subseteq"$. Thus it can be realized geometrically (see \cite[Appendix A.2]{Davis}). We denote this by $\Sigma(W,S):=|\mathcal{WS}|$ and call it `the' \textit{Davis complex}. The Coxeter group $W$ acts by left multiplication properly discontinuously and cocompactly via isometries on the Davis complex $\Sigma(W,S)$. The fundamental domain for this action is $|\mathcal{S}|$ and furthermore $\Sigma(W,S)$ is simply connected.

\begin{example}
	Consider the dihedral group $D_3=\langle s,t| s^2,t^2, (st)^3\rangle$. The spherical subsets of $S=\{s,t\}$ are $\emptyset,\{s\}, \{t\}, S$. Thus our fundamental domain looks like this:
	
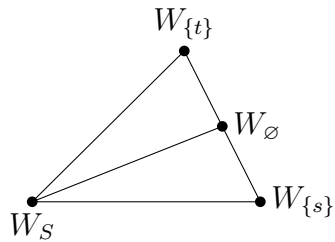
\begin{center}
		\begin{tikzpicture}
			\coordinate[label=below: {$W_S$}] (A) at (0,0);
			\coordinate[label=right: {$W_{\{s\}}$}] (B) at (3,0);
			\coordinate[label=above: {$W_{\{t\}}$}] (C) at (2,2);
				\coordinate[label=right: {$W_\emptyset$}] (D) at (2.5, 1);

			\fill (A) circle (2pt);
			\fill (B) circle (2pt);
			\fill (C) circle (2pt);
			\fill (D) circle (2pt);
			
			\draw (A) -- (B);
			\draw (A) -- (C);
			\draw (C) --(D) ;
			\draw (D) -- (B);
			\draw (A) -- (D);
		\end{tikzpicture}
	\end{center}
Now considering all the cosets we obtain the following picture of the full Davis-complex:
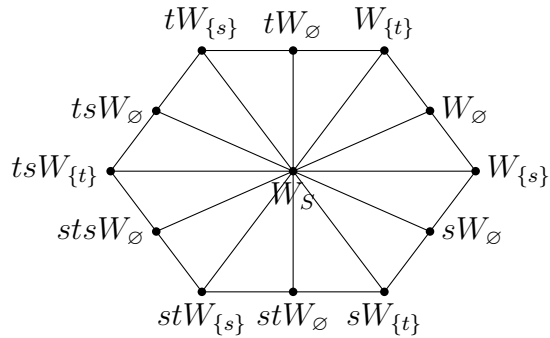
\begin{center}
	\begin{tikzpicture}[scale=0.8]
				\coordinate[label=below: {$W_S$}] (A) at (0,0);
		\coordinate[label=right: {$W_{\{s\}}$}] (B) at (3,0);
		\coordinate[label=above: {$W_{\{t\}}$}] (C) at (1.5,2);
		\coordinate[label=right: {$W_\emptyset$}] (D) at (2.25, 1);
		\coordinate[label=above: {$tW_\emptyset$}] (E) at (0,2);
		\coordinate[label=above: {$tW_{\{s\}}$}] (F) at (-1.5,2);
		\coordinate[label=left: {$tsW_{\emptyset}$}] (G) at (-2.25,1);
		\coordinate[label=left: {$tsW_{\{t\}}$}] (H) at (-3,0);
		\coordinate[label=left: {$stsW_{\emptyset}$}] (I) at (-2.25,-1);
		\coordinate[label=below: {$st W_{\{s\}}$}] (J) at (-1.5,-2);
		\coordinate[label=below: {$stW_\emptyset$}] (K) at (0,-2);
		\coordinate[label=below: {$sW_{\{t\}}$}] (L) at (1.5,-2);
		\coordinate[label=right: {$sW_\emptyset$}] (M) at (2.25,-1);
		
		\fill (A) circle (2pt);
		\fill (B) circle (2pt);
		\fill (C) circle (2pt);
		\fill (D) circle (2pt);
		\fill (E) circle (2pt);
		\fill (F) circle (2pt);
		\fill (G) circle (2pt);
		\fill (H) circle (2pt);
		\fill (I) circle (2pt);
		\fill (J) circle (2pt);
		\fill (K) circle (2pt);
		\fill (L) circle (2pt);
		\fill (M) circle (2pt);
		
		\draw (A) -- (B);
		\draw (A) -- (C);
		\draw (C) --(D) ;
		\draw (D) -- (B);
		\draw (A) -- (D);
		\draw (A) -- (E);
		\draw (A) -- (F);
		\draw (A) -- (G);
		\draw (A) -- (H);
		\draw (A) -- (I);
		\draw (A) -- (J);
		\draw (A) -- (K);
		\draw (A) -- (L);
		\draw (A) -- (M);
		
		\draw (C) -- (E);
		\draw (E) -- (F);
		\draw (F) -- (G);
		\draw (G) -- (H);
		\draw (H) -- (I);
		\draw (I) -- (J);
		\draw (J) -- (K);
		\draw (K) -- (L);
		\draw (L) -- (M);
		\draw (M) -- (B);
	\end{tikzpicture}
\end{center}
\end{example}

This simplicial structure is great for some applications, however for the question about the \CAT\ property, it is not the best, especially since there are many ways to geometrically realize a simplicial complex yielding different metrics. We will thus define a new metric, the so called \textit{Moussong-metric} on $\Sigma(W,S)$ using the combinatorics of the Coxeter group. Note that since $|\mathcal{S}|$ is a fundamental domain for the natural action of $W$ on $\Sigma(W,S)$, it suffices to define a new metric on this fundamental domain. Therefore we can assume without loss of generality that $(W,S)$ is finite for now. We define a convex polyhedron in $\mathbb{R}^{|S|}$. Since $(W,S)$ is finite, the cosine matrix $A=(a_{ij})_{i,j\in\{1,...,n\}}$ is positive definite and therefore defines a scalar product on $\mathbb{R}^{|S|}$ and thus a metric. Since $A$ is positive definite, we can find vectors $u_1,...,u_n$ such that $\langle e_i,u_j\rangle =\begin{cases}
    1 \quad \text{ if } i=j,\\
    0 \quad \text{ else.}
\end{cases}$ for all $i,j\in \{1,...,n\}$. For $T\subseteq S$ let $C_T$ denote the cone spanned by the vectors $\{u_i| i\in T\}$.  Let $p$ denote the unique point in the interior of $C_S$ at distance $1$ from all cones $C_{S-\{s\}}$, $s\in S$ with respect to $A$. We can orthogonally project $p$ to each of these cones using the scalar product induced by $A$. Let $q_T$ denote the image of the orthogonal projection of $p$ onto $C_T$. Note that $q_S=p$ and $q_\emptyset=o$ where $o$ is the origin. We set $C={\rm conv}(\{q_T\mid T\subseteq S\})$ as the convex hull of all the points $q_T$. The convex cone $C$ equipped with the metric induced by the positive definite bilinear form $\langle.,.\rangle_A$ defines the \textit{Moussong-metric} on $|\mathcal{S}|$.


\begin{example}
	We once again consider the dihedral group $D_3$. Its cosine matrix is given by $A=\begin{pmatrix}
		1 & -0.5\\
		-0.5 & 1
	\end{pmatrix}$. To construct the cone $C_S$ we first need the vectors $u_t$ and $u_s$. An explicit calculation shows $u_s=(\frac{2}{3}, \frac{4}{3})$ and $u_t=(\frac{4}{3},\frac{2}{3})$. This then gives us the following picture.
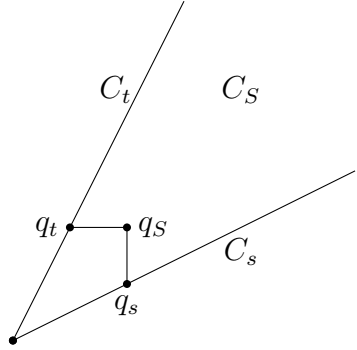
\begin{center}
	\begin{tikzpicture}[scale=0.75]
		\coordinate (A) at (0,0);
		\coordinate (B) at (6,3);
		\coordinate (C) at (3,6);
		\coordinate (D) at (0,0);
		\coordinate[label=right: {$q_S$}] (E) at (2, 2);
		\coordinate[label=below: {$q_s$}] (F) at (2,1);
		\coordinate[label=left: {$q_t$}] (G) at (1,2);
		\coordinate[label=below: {$C_s$}] (H) at (4,2);
		\coordinate[label=above: {$C_t$}] (I) at (1.8,4);
		\coordinate[label= above: {$C_S$}] (K) at (4,4);
		
		\fill (A) circle (2pt);
		\fill (D) circle (2pt);
		\fill (E) circle (2pt);
		\fill (F) circle (2pt);
		\fill (G) circle (2pt);
		
		\draw (A) -- (B);
		\draw (A) -- (C);
		
		\draw (E) -- (G);
		\draw (E) -- (F);
	\end{tikzpicture}
\end{center}
\end{example}

From now on we will consider $\Sigma(W,S)$ with this metric. Note that essentially the same construction can be made in hyperbolic space rather than Euclidean space. There is also a different way of constructing the same metric, using so called Coxeter polytopes, see \cite[Chapter 7]{Davis}.

\begin{proposition}\label{Proposition}(\cite[Chapters 12, 13]{Moussong})
	Let $(W,S)$ denote a Coxeter system and $\Sigma(W,S)$ denote the Davis complex with the Moussong metric. Then
	\begin{enumerate}
		\item $\Sigma(W,S)$ is simply connected.
		\item $W$ acts on $\Sigma(W,S)$ cocompactly and properly discontinuously via isometries.
		\item The link of a simplex corresponding to a subset $I\subseteq S$ is isometric to $lk(B_I,N(A))$.
	\end{enumerate}
\end{proposition}
This sums up everything we need to know about the Davis complex. In the next section we will discuss how this information is useful and how the original argument worked.

\section{The gap in the original Argument}
Let us first recall three very important theorems that are at the heart of the proof that $\Sigma(W,S)$ is indeed a \CAT\ space. It is helpful to understand the \CAT\ case first, since the Gromov-hyperbolic case is very similar.

In general, it is very hard to check if a geodesic metric space satisfies the \CAT\ condition, since once needs to check all geodesic triangles for the condition. Luckily for complete, simply connected spaces it suffices to check the condition locally. More precisely the so called ``Cartan-Hadamard-Theorem'' says the following.
\begin{theorem}(\cite[Thm II.4.1]{BH})
	A simply connected complete metric space $X$ is \CAT\ if and only if it is locally \CAT.
\end{theorem}
However, even checking the condition locally at every point is a daunting task. Thus the case of polyhedral complexes has been studied and a simpler condition has been found when these spaces are \CAT.
\begin{theorem}(\cite[Thm II.5.4]{BH})
	Given $\kappa\leq 0$, an $M_\kappa$ polyhedral complex with finitely many shapes of cells is locally CAT$(\kappa)$ if and only if the link of every vertex is CAT$(1)$.
\end{theorem}
For almost negative matrices and their nerve complexes, one only needs to calculate their girth to check if they are CAT$(1)$ by the following result.
\begin{proposition}(\cite[p. 17]{Moussong}/\cite[Thm. II.5.4]{BH})
	 The nerve complex of an almost negative matrix is CAT$(1)$ if and only if its girth and the girth of all links is $\geq 2\pi$.
\end{proposition}
A big part of \cite{Moussong} is calculating the girth of nerve complexes of almost negative matrices. We will not give the arguments here. However for the \CAT\ case, the main result is the following
\begin{proposition}(\cite[Cor. 10.2]{Moussong})
	Given a Coxeter system $(W,S)$ with cosine matrix $A$, then $g(N(A))\geq 2\pi$ and $g(N(lk(I,A)))\geq 2\pi$ for all $I$ such that $A_I$ is positive definite.
\end{proposition}
Since every link of every cell in $\Sigma(W,S)$ is isometric to a link of a cell in the fundamental domain joined with a sphere \cite[p. 41]{Moussong}, the above theorems together with Proposition 2.2 show that indeed, for every Coxeter system $(W,S)$, the complex $\Sigma (W,S)$ is a \CAT\ space and thus every Coxeter group is a \CAT\ group. This part of the argument is problem free (see \cite[Thm. 14.1]{Moussong}).

For the Gromov-hyperbolic case one realizes the following: The hyperbolic space $\mathbb{H}^n$ is near-isometric to the Euclidean space $\mathbb{E}^n$ close to the origin, in fact, the map gets arbitrarily close to an isometry, if we get close enough to the origin. Now suppose we are given a Coxeter system $(W,S)$, such that in Proposition 3.4, we have a strict inequality for the girths of the nerve and all its links. Then we can construct the metric on the Davis complex choosing a point $p$ at distance $\varepsilon$ from all the cones instead of at distance $1$ and use hyperbolic space instead of Euclidean. Since the hyperbolic space is near-isometric to Euclidean space, the links of all cells will also be near-isometric to the Euclidean links. Since we have the strict inequality, and the number of spherical subsets of $S$ is finite, we can choose an $\varepsilon$ so small that the girth of all the links is still $\geq 2\pi$ (because of the strict inequality in Proposition 3.4). Therefore, if we can find conditions on $(W,S)$, such that the strict inequality holds, we have proven that those Coxeter groups are CAT$(-1)$ groups, and thus Gromov-hyperbolic \cite[p. 50]{Moussong}.

A big part of the calculations of the girth of nerve complexes in \cite{Moussong} deals exactly with this strict inequality. Many lemmas and propositions have an additional sentence of the form ``if additionally $\ast $ holds, then the strict inequality holds''. The conditions in $\ast$ vary, however they all aim at making sure that the Coxeter system $(W,S)$, from which the almost negative matrix $A$ is obtained via the cosine matrix construction, does not contain a subgroup isomorphic to $\mathbb{Z}^2$. The main results are \cite[Lemma 10.3 and Cor. 10.4]{Moussong} for which we don't have a counterexample. However their proof relies crucially on \cite[Lemma 9.11]{Moussong}, for which we give a counterexample. In the final section we fix the proof of \cite[Lemma 10.3]{Moussong}, which implies \cite[Cor. 10.4]{Moussong}. In this section we discuss the problems with the proof of \cite[Lemma 9.11]{Moussong}, which come from two previous lemmas. We will always mention the statement of the original lemma for better readability as well as state in which cases it does not hold.

Before we can discuss the lemmas we need a few very important definitions.
\begin{definition}
	Let $A=(a_{ij})_{i,j\in I}$ denote a symmetric matrix. A \textit{principal submatrix} of $A$ is a matrix of the form $B= (a_{ij})_{i,j\in J}$ for a subset $J\subsetneq I$. A matrix $A$ is called \textit{parabolic} if det$(A)=0$, $|I|\geq 2$ and all principal submatrices are positive definite.
\end{definition}
\begin{remark}\label{remark}
    We require $|I|\geq 2$ for the parabolicity for the following reasons:
    \begin{enumerate}
        \item With this definition, in the case of Coxeter groups, the cosine matrix is parabolic if and only if the Coxeter group is affine. 
        \item Due to the problems occuring in \cite[Lemma 9.7]{Moussong}, we need this definition, as otherwise dealing with the case that a link matrix has a $0$ on the diagonal is very tricky.
        \item It illustrates the process how the counterexamples were found and why we are working ``from the bottom up'', that is starting with small link matrices and working our way up.
    \end{enumerate}
\end{remark}
Now we can state the faulty Lemmas and prove in which cases they still hold true.
\vspace{0.4cm}
\begin{ftheo} Let $A$ denote an almost negative $n\times n$ matrix and $u\in\mathbb{R}^n$ a vector with nonnegative coordinates and with $\langle u,u\rangle =1$ (with respect to the bilinear form given by $A$). Then there exists a $z\in N(A)$ such that $\langle u,z\rangle \geq 1$. 
Suppose $A$ has no principal parabolic submatrices, then there exists a vector $z \in N(A)$ such that $\langle u,z\rangle >1$.
\end{ftheo}
\vspace{0.2cm}
\begin{lemma}
	Let $A$ denote an almost negative matrix. Then the strict inequality \cite[Lemma 9.5]{Moussong} holds, if $A$ does not have a principal submatrix of order $\geq 2$ which contains a row or column consisting entirely of zeroes and there exists a counterexample to the general statement.
\end{lemma}
\begin{proof}
We first give a counterexample to \cite[Lemma 9.5]{Moussong}:

Consider the matrix $A=\begin{pmatrix}
	1&0\\
	0&0
\end{pmatrix}$ and $\mathbb{R}^2$ with the basis ${(1,0), (0,1)}$. Then the nerve consists of exactly one point: $N(A)=p$, with $p={(1,0)}$.\\
Consider the point $u=(1,1)\in \mathbb{R}^2$, which has nonnegative coordinates. There is only one possible choice for $z\in N(A)$, namely $z=(1,0)$, so we have $\langle u,z\rangle=\langle (1,1),(1,0)\rangle=1\not{>}1$.

The original proof works in every other case, since it only breaks in the second to last line on page 27 in \cite{Moussong}. It is claimed that $\langle w,w\rangle <0$, but this is false in the above case for $w=(0,1)$.
But it holds true if we do not have a principal submatrix of order $\geq 2$ having a row or column consisting entirely of zeroes. This is true because of the following reason. First of all by filling up the vectors with $0$s it follows that if any principal submatrix provides a counterexample, then the entire matrix provides a counterexample. Hence we assume $A$ does not have a principal submatrix of order $\geq 2$ which has a row or column consisting entirely of zeroes and work in the notation of Moussong. Since $u\notin N(A)$, at least one $\alpha_w$ for $w\in F'\cap (F_2\cup F_3)$ is not equal to $0$. If this is the case for a $w\in F_3$, then $\langle w,w\rangle <0$ since $A$ does not contain principal parabolic submatrices. Hence there is at least one $w\in F'\cap F_2$ such that $\alpha_w\neq 0$. Since $w\in F'\cap F_2$, $w=u_i$ for some $i$ and thus $a_{ii}=0$ if $\langle w,w\rangle=0$. If there are $w\neq w'\in F'\cap F_2$ such that $\alpha_w\neq 0\neq \alpha_{w'}$, then the entry $a_{ij}$ which corresponds to $(w,w')$ is $0$ and thus the principal submatrix $A_{\{i,j\}}$ consists entirely of zeroes. So suppose there is exactly one such $w$ and $\langle w,w\rangle =0$. Now we know that $u=\alpha_w w+\beta z$ for $z\in N(A)$ and $\beta >0$. We have 
$$ 1=\langle u,u\rangle=\langle w,w\rangle +2\langle w,\beta z\rangle +\beta^2\langle z,z\rangle=2\beta \langle w,z\rangle +\beta^2\langle z,z\rangle$$
By definition of $w,z$ we have $\langle w,z\rangle\leq 0$. If we have equality, then the principal submatrix corresponding to the index set of the non-zero summands of $z$ and $w$ gives us a principal submatrix which has a row consisting of zeroes. If $\langle w,z\rangle <0$, then $\beta >1$ and hence $\langle u,z\rangle = \langle w,z\rangle +\beta \langle z,z\rangle =\beta >1$ as desired.
\end{proof}
\begin{remark}
    If one defines a $1\times 1$-matrix with entry $0$ as parabolic, then $A$ cannot contain a principal submatrix of order $\geq 2$ which has a row consisting entirely of zeroes, so the Lemma holds true. However the following two Lemmas also break for this definition, in that case the mistake happens in the proof of Lemma 9.7.
\end{remark}
\begin{ftheo1}
    
Let $A$ denote an almost negative $n\times n$ matrix and $x,y\in N(A)$ such that $\langle x,y\rangle >-1$. If $A$ has no principal parabolic submatrices of order $\geq 3$ and if there is no simplex in $N(A)$ containing both $x$ and $y$, then $d(x,y)<\cos^{-1}\left(\langle x,y\rangle\right)$, where $d$ denotes the intrinsic metric of $N(A)$ as in \cite[Chapter 3]{Moussong} or \cite[Chapter 1.3]{P}.
\end{ftheo1}
\begin{lemma}\label{9.7}
	Let $A$ denote an almost negative matrix. Then the strict inequality in \cite[Lemma 9.7]{Moussong} holds if $A$ does not have a link matrix which has a principal submatrix of order $\geq 2$ with a row or column consisting entirely of zeroes and if no link matrix contains a principal parabolic submatrix of order $2$. Moreover there exists a counterexample to the general statement.
\end{lemma}

\begin{proof}
	Once again we first give the counterexample.
	
In this case we consider the matrix $A=\begin{pmatrix}
	1&0&-1\\
	0&1&0\\
	-1&0&1
\end{pmatrix}$. This matrix is not parabolic because the matrix $A_I$ for $I=\{1,3\}$ is not positive definite. Again we use the standard basis $\{e_1,e_2,e_3\}$ of $\mathbb{R}^3$. We can see that the nerve $N(A)$ consists of three points and two edges of length $\frac{\pi}{2}$.
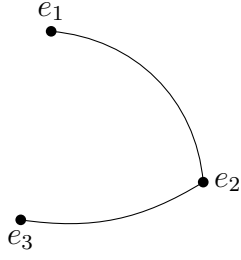
\begin{center}
\begin{tikzpicture}
	\coordinate[label=above: $e_1$] (A) at (0,2);
	\coordinate[label=right: $e_2$] (B) at (2,0);
	\coordinate[label=below: $e_3$] (C) at (-0.4,-0.5);
	
	\draw (A) to[bend left=40] (B); 
	\draw (B) to[bend left=20] (C); 
	\fill (A) circle (2pt);
	\fill (B) circle (2pt);
	\fill (C) circle (2pt);
	
\end{tikzpicture}
\end{center}
We define $x=(1,0,0)$ and $y=(0,\sin\phi,\cos\phi)$ for $\phi \in \left(0,\frac{\pi}{2}\right)$.\\
Let's quickly check that $y\in N(A)$:
\begin{equation*}
\begin{split}
\langle y,y\rangle &=(0,\sin\phi,\cos\phi)A(0,\sin\phi,\cos\phi)^t\\
&=(-\cos\phi,\sin\phi,\cos\phi)\cdot (0,\sin\phi,\cos\phi)^t 
=\sin^2\phi+\cos^2\phi=1
\end{split}
\end{equation*}
Additionally, we have:
$$\langle x,y\rangle =(1,0,-1)\cdot (0,\sin\phi,\cos\phi)=-\cos\phi.$$
Their distance with respect to the intrinsic metric is given by $\frac{\pi}{2}$ plus the length of the arc connecting $e_2$ and $y$ which is given by $\phi$. Now choose the value $\phi=\frac{\pi}{4}$, then we have: 
$$cos^{-1}\left(-\cos\left(\phi\right)\right)=\frac{3\pi}{4}\quad \text{and } \frac{\pi}{2}+\frac{\pi}{4}=\frac{3\pi}{4}$$
So the left side and the right side are equal, contradicting the statement of the Lemma.

There are two problems in the proof. One is that \cite[Lemma 9.5]{Moussong} is applied to $lk(I,A)$, the other one is in the statement that if $A$ does not contain principal parabolic submatrices of order $\geq 3$ then $lk(I,A)$ does not contain any principal parabolic submatrices. This is false, e.g. the matrix $A=\begin{pmatrix}
    1 &0&0 \\
    0&1&-1 \\
    0&-1&1\end{pmatrix}$ has the matrix $\begin{pmatrix}
        1&-1\\
        -1&1
    \end{pmatrix}$ as link matrix. In Step 3 of the proof of Lemma \ref{split} we will see that if $lk(I,A)$ contains a principal parabolic submatrix of order $\geq 3$, then so does $A$. Therefore the only critical cases are if a $2\times2$ principal parabolic submatrix appears and cases in which \cite[Lemma 9.5]{Moussong} can fail, which we excluded.
\end{proof}
We move on to the, in some sense, most critical lemma:
\begin{lemma}\label{9.11}
Let $A$ denote an almost negative matrix. Then the strict inequality \cite[Lemma 9.11]{Moussong} holds if $A$ does not have a link matrix which has a principal submatrix of order $\geq 2$ with a row or column consisting entirely of zeroes and if no link matrix contains a principal parabolic submatrix of order $2$. Moreover there exists a counterexample to the general statement.
\end{lemma}
In this case, the original statement was:
\newpage
\begin{ftheo2}
    
	Let $A$ denote an almost negative matrix and suppose there is a vertex $v$ of $N=N(A)$ which is connected to every other vertex of $N(A)$ by an edge. Let $x,y\in St(v,N)\setminus Ost(v,N)$ such that $d(x,y)<\pi$, where $d$ denotes the intrinsic metric of $N(A)$.\\
If $A$ has no principal parabolic submatrices of order $\geq 3$ and no simplex of $St(v,N(A))$ contains both $x$ and $y$, then $d'(x,y)<d(x,y)$ where $d'$ denotes the intrinsic metric of the suspension $Slk(v,N(A))$. 
\end{ftheo2}
\vspace{0.3cm}
Let us first recall the definitions used in this lemma. In this context $St(v, N)$ denotes the star of a vertex, i.e. the union of all cells containing $v$ and $Ost(v,N)$ is the open star, that is, the union of all interiors of cells containing $v$. The suspension $SK$ of a spherical complex $K$ is the join of $K$ with $\mathbb{S}^0$. By \cite[Lemma 8.2 (ii)]{Moussong} it follows that $St(v,N)\subseteq Slk(v,N)$ which is implicitly used in the statement of the Lemma.
\begin{proof} We once again start with the counterexample.
	
		In this case we need a $4\times 4$ matrix to provide a counterexample. This counterexample once again has the ``original'' $2\times 2$ matrix as a link. The matrix is given by 
	$$A=\begin{pmatrix}
		1&0&0&-1\\
		0&1&0&0\\
		0&0&1&0\\
		-1&0&0&1
	\end{pmatrix}$$
	This satisfies the requirements. Once again choosing the standard basis $\{e_1,e_2,e_3,e_4\}$, we can see that the nerve $N(A)$ consists of two equilateral spherical triangles glued together along the edge connecting $e_2$ and $e_3$. All the edges have length $\frac{\pi}{2}$.
	
	
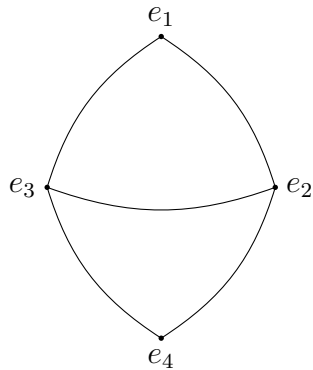
\begin{center}
		\begin{tikzpicture}[scale=0.5]
			\coordinate[label=above: {$e_1$}] (A) at (0,4);
			\coordinate[label=right: {$e_2$}] (B) at (3,0);
			\coordinate[label=left: {$e_3$}] (C) at (-3,0);
			\coordinate[label=below: {$e_4$}] (D) at (0,-4);
			
			\fill (A) circle (2pt);
			\fill (B) circle (2pt);
			\fill (C) circle (2pt);
			\fill (D) circle (2pt);
			
			\draw (A) to[bend left=20] (B);
			\draw (A) to[bend right=20] (C);
			\draw (B) to[bend left=20] (C);
			\draw (D) to[bend right=20] (B);
			\draw (D) to[bend left=20] (C);
		\end{tikzpicture}
	\end{center}
	
	Now consider the vertex $v=e_2$. Then we can see that $St(v,N(A))\setminus Ost(v,N(A))$ consists of three points and two edges of lengths $\frac{\pi}{2}$, the precise complex we dealt with in the previous lemma and counterexample with different labeling.
	
	Now we choose ``the same points'', meaning $x=(1,0,0,0)$ and $y=(0,0,\cos\phi,\sin\phi)$ for $\phi=\frac{\pi}{4}$ for example.
	
	The link $lk(v,N(A))$ is the same complex as $St(v,N(A))\setminus Ost(v,N(A))$. The complex $Slk(v,N(A))$ has a special structure, namely we take two copies of $N(A)$ and glue them together along $lk(v,N(A))=St(v,N(A))\setminus Ost(v,N(A))$. Additionally, we can see that $N(A)=Clk(v,N(A)):=1\ast lk(v,N(A))$, where $Clk(v,N(A))$ denotes the spherical cone of $lk(v,N(A))$ which is defined as the spherical join of $lk(v,N(A))$ with the north pole of a sphere.
	
	Therefore, using the definition of the intrinsic metric (or geometry on the sphere), we can see that $d(x,y)=\frac{3\pi}{4}<\pi$. Additionally, the same Lemma and idea tell us that $d'(x,y)=\frac{3\pi}{4}$, which means there distance in the suspension is the same as in the original complex, contradicting the statement of the Lemma.

    The proof in \cite{Moussong} considers two cases. In the first case, there exists a simplex $B$ in $N(A)$ which contains both $x$ and $y$. This case is proven using \cite[Lemma 9.7]{Moussong} and hence works for all cases in which the Lemma holds. In the second case there does not exist a simplex which contains both $x$ and $y$. The original proof uses an induction argument. The induction does not seem to rely on the faulty Lemma at first, however it is used implicitly to deal with the case that the geodesic connecting $x$ and $y$ is contained in $St(v,N(A))$.

    Hence we give a proof for the case that the geodesic connecting $x$ and $y$ is contained in $St(v,N(A))$ relying on \cite[Lemma 9.7]{Moussong} and also explain the rest of the proof for completeness.

    Let $n$ denote the number of vertices of $N(A)$. The Lemma is true for $n=0,1,2$ by elementary calculations. So we assume we have proven the Lemma for all nerves of almost negative matrices with number of vertices strictly smaller than $n$. 

    Let $p\colon [0,d(x,y)]\to N(A)$ denote a geodesic in $N(A)$ connecting $x$ and $y$. Let $P$ denote the image of $p$.\\
    First suppose that $P$ is contained in $St(v,N(A))$. In that case an application of \cite[Lemma 9.8]{Moussong} yields that $P$ is contained in $St(v,N(A))-Ost(v,N(A))$, because else it would have length $\geq \pi$ contradicting $d(x,y)<\pi$. Since $p$ is connected and $\langle .,.\rangle$ is continuous, we can find points $z_1,...,z_n$ for $0< t_1 <...<t_n<d(x,y)$ such that $\langle x,z_1\rangle>-1$, $\langle z_i,z_{i+1}\rangle >-1$ (for all $i\in\{1,...,n-1\}$) and $\langle z_n,y\rangle >-1$. Moreover, we can choose the $z_i$ such that $x$ and $z_1$ don't lie in a common simplex either, since any two points contained in a simplex have distance $<\pi$ and within every simplex, the metric is given by $d(a,b)=\cos^{-1}\langle a,b\rangle$ (compare \cite[Cor. 9.9 and 9.10]{Moussong}). Since $d(x,y)=d(x,z_1)+d(z_1,z_2)+...+d(z_{n-1},z_n)+d(z_n,y)$, we can argue exactly as in case 1 of Moussong's proof (\cite[Lemma 9.11]{Moussong}), that is, apply \cite[Lemma 9.7]{Moussong} to $Slk(v,N(A))$ for each pair of points $(x,z_1)$, $(z_i,z_{i+1})$ and $(z_n,y)$ and obtain the desired (strict) inequality. \\
    Now suppose $p$ is not contained in $St(v,N(A))$. Using the above argument, we can assume that $p(t)\notin St(v,N(A))$ for all $t\in (0,d(x,y))$.
    Let $G$ and $H$ denote the vertex sets of the supports of $x$ and $y$ respectively. We want to argue that the set $G\cup H\cup \{v\}$ is strictly smaller than the vertex set of $N$. Consider $Supp\left(p\left(\varepsilon\right)\right)$ for a very small $\varepsilon>0$. The support of $p(\varepsilon)$ is not a subset of $G\cup H$, since by assumption $G\cup H$ does not span a simplex. Moreover we have $p(\varepsilon)\notin St(v,N(A))$ and hence $Supp\left(p(\varepsilon)\right)\not\subseteq G\cup H\cup \{v\}$, therefore we conclude that $G\cup H \cup \{v\}$ is a proper subset of the vertex set of $N(A)$ and hence the induction hypothesis applies.

    Hence the proof is valid assuming the conditions on $A$ from Lemma \ref{9.7}, which is what we wanted to show.
\end{proof}
Thus, the only cases where the proof of \cite[Thm. B]{Moussong} fails are cases where the cosine matrix of the Coxeter system $(W,S)$ has a problematic link matrix, meaning the cases discussed in Lemma \ref{9.11} 

All the counterexamples stem from the same ``problem matrix'' used as a counterexample for the first lemma, arising as a link matrix. This is how these examples were found and suggests that the approach ``from the bottom up'', i.e. starting with a link matrix and reconstructing the original matrix might be a good approach. In the next chapter we will discuss this in an example before proving the main technical lemma.
\section{The characterization of hyperbolic Coxeter groups}
We will now consider the cases in which the original argument contains a gap. To do so, we will show that even in the potentially problematic cases \cite[Lemma 10.3]{Moussong} holds. This boils down to a technical result about these problematic matrices. 

Before proving the lemma, we will discuss an example in ``both directions" to motivate the upcoming calculations.
\begin{example}
	Let's start with the easier direction. Start with the matrix $$A=\begin{pmatrix}
		1 & 0&0&-1\\
		0&1&0&0\\
		0&0&1&0\\
		-1&0&0&1
	\end{pmatrix}.$$ For $I=\{1,2\}$ we will now calculate $lk(I,A)$:

First we need to understand what $U^\perp$ looks like where $U=\langle e_1,e_2\rangle$. Since $(v_1,v_2,v_3,v_4)\cdot A=(v_1-v_4,v_2,v_3,v_4-v_1)$ for any vector $(v_1,v_2,v_3,v_4)\in \mathbb{R}^4$, we can deduce that $U^\perp=\langle (1,0,0,1), (0,0,1,0)\rangle$. Thus, if $\varphi\colon \mathbb{R}^4\to U^\perp$ denotes the orthogonal projection, we can see that $\varphi(e_3)=e_3$ and $\varphi(e_4)=(1,0,0,1)$ (because $e_1\perp (1,0,0,1)$ and $(1,0,0,1)-e_1=e_4$). Now $(1,0,0,1)\cdot A=(0,0,0,0)$. Thus the link matrix will have a $0$-column, more precisely, an easy calculation shows that $lk(I,A)=\begin{pmatrix}
	1&0\\
	0&0
\end{pmatrix}$ which is the exact matrix we used as a counterexample to \cite[Lemma 9.5]{Moussong}. 

On the other hand we can ask the question what a $3\times 3$ almost negative matrix $A$ needs to look like so that it has this $2\times 2$ matrix as a link matrix. For simplicity we will assume that all diagonal entries are $1$ and we compute $lk(\{1\},A)$. In general $$A=\begin{pmatrix}
	1 & a_{12} & a_{13}\\
	a_{12} & 1 & a_{23}\\
	a_{13} & a_{23} &1
\end{pmatrix}$$ since $A$ is symmetric. First we need to see which vectors are orthogonal to $e_1$. Since the vector space $U^\perp$ is $2$ dimensional, it suffices to find two linearly independent vectors to describe $U^\perp$. For example the vectors $(-a_{12},1,0)$ and $(-a_{13},0,1)$ form a basis of $U^\perp$. Now we need to orthogonally project $\varphi\colon \mathbb{R}^3\to U^\perp$ for the vectors $e_2$ and $e_3$. One way of looking at this is figuring out when $e_2+r\cdot e_1\in U^\perp$ for $r\in \mathbb{R}$, since $e_1$ is orthogonal to $U^\perp$. This is exactly the case when $r=-a_{12}$, or in other words $\varphi(e_2)=(-a_{12},1,0)$. Similarly we obtain $\varphi(e_3)=(-a_{13}, 0 , 1)$. The link matrix then has the form $\begin{pmatrix}
b_{11} & b_{12}\\
b_{12} & b_{22}
\end{pmatrix}$. A straightforward calculation now shows $b_{11}=1-a_{12}^2$, $b_{12}=a_{23}-a_{12}a_{13}$ and $b_{22}=1-a_{13}^2$. 

Suppose we want to obtain $b_{12}=0=b_{22}$, then we need $a_{13}=-1$, since $a_{13}\leq 0$. And additionally we have $0=a_{23}+a_{12}$ or in other words $a_{23}=-a_{12}$ and since both $a_{23}$ and $a_{12}$ are non-positive, we can deduce $a_{12}=0=a_{23}$ and thus 
$$A=\begin{pmatrix}
	1 &0&-1\\
	0&1&0\\
	-1&0&1
\end{pmatrix}$$
which is the matrix used to provide the counterexample to \cite[Lemma 9.7]{Moussong}.
\end{example}
\begin{lemma}\label{split}
	Let $A$ denote an almost negative $n\times n$ matrix and suppose that there exists a link matrix which has a row consisting entirely of zeroes. Then either $A$ is parabolic or, after a permutation of indices, we can write $A$ as a block matrix $A=\begin{pmatrix}
		A_1 & 0\\
		0 &A_2
	\end{pmatrix}$. 
\end{lemma}
If we can (after a permutation of indices) write a matrix $A$ as above, we call $A$ \textit{reducible}.
\begin{proof}
		Suppose we are given a symmetric $m\times m$ matrix $C=(c_{ij})$ with $c_{11}>0$. In this proof $C$ is an arbitrary almost negative matrix and $A$ is the matrix in the Lemma. We prove the Lemma in three steps, first we prove a formula to calculate links. Secondly we deal with the case that a $2\times 2$ or bigger link matrix has a $0$-column and finally we treat the case where a $1\times 1$ link matrix with entry $0$ appears.
		
	\underline{Step 1:} Compute $lk\left({1},C\right)=\left(d_{ij}\right)$:
	
	Start with $U=\langle e_1\rangle$ where $e_1$ denotes the first standard basis vector. Write $U^\perp=\langle u_2,...,u_m\rangle$ where one possible choice for $u_j$ is $u_j=\left(-c_{1j},0,...,0,c_{11},0,...,0\right)$. To compute $\phi_U\left(e_j\right)$ ($j\in\{2,...,m\}$) we therefore need to find the solution to the linear equation $e_j=r\cdot e_1+s\cdot u_j$. This yields $\phi_U\left(e_j\right)=-\frac{c_{1j}}{c_{11}}e_1+e_j$. Then we have for $i,j\in\{2,...,m\}$ and $\langle .,.\rangle$ induced by $C$:
	\begin{equation*}
		\begin{split}
			d_{ij}&=\langle \phi_U\left(e_i\right),\phi_U\left(e_j\right)\rangle=\left\langle -\frac{c_{1i}}{c_{11}}e_1+e_i,-\frac{c_{1j}}{c_{11}}e_1+e_j\right\rangle\\
			&=\frac{c_{1i}c_{1j}}{c_{11}^2}\langle e_1,e_1\rangle -\frac{c_{1i}}{c_{11}}\langle e_1,e_j\rangle -\frac{c_{1j}}{c_{11}}\langle e_i,e_1\rangle +\langle e_i,e_j\rangle\\
			&=\frac{c_{1i}c_{1j}}{c_{11}^2} c_{11}-\frac{c_{1i}}{c_{11}}c_{1j}-\frac{c_{1j}}{c_{11}}c_{1i}+c_{ij}\\
			&=-\frac{c_{1i}c_{1j}}{c_{11}}+c_{ij}
		\end{split}
	\end{equation*}
	This formula will turn out to be very useful for our further analysis since we can compute a link $lk(J,A)$ ``step by step''. In particular we have $lk(J,A)=lk\left(\{j\},lk\left(J\setminus\{j\},A\right)\right)$ for any $j\in J$ (see Lemma \ref{1.3}).
	
	\underline{Step 2:}  $\#J\leq n-2$: We show that $A$ is reducible.
	
	So now suppose that $lk(J,A)=\left(f_{rs}\right)$ for $\#J\leq n-2$ contains a $0$-column. That means there exists an $i\in \{1,...,n\}\setminus J$ such that $f_{ik}=0$ for all $k\in \{1,...,n\}\setminus J$. We choose a $j\in J$ and then take a look at $lk\left(J\setminus\{j\},A\right)=\left(b_{rs}\right)$. Using the formula for the links we can deduce that
	$$0=f_{ik}=-\frac{b_{ji}b_{jk}}{b_{jj}}+b_{ik} \quad\text{ for all }k\in \{1,...,n\}\setminus J.$$
	Since we started with an almost negative matrix, all link-matrices are almost negative, too (Lemma \ref{1.3}). So for $k\notin\{i,j\}$ we can deduce that $b_{ik}=0$ and either $b_{jk}=0$ or $b_{ji}=0$.\\
	If $b_{ji}=0$ then $lk\left(J\setminus\{j\},A\right)$ is reducible for the partition of the index set given by $\{1,...,n\}\setminus \left(J\setminus \{j\}\right) \setminus \{i\}$ and $\{i\}$.\\
	If $b_{jk}=0$ for all $k \notin \{i,j\}$ then $lk\left(J\setminus\{j\},A\right)$ is reducible for the partition of the index set given by $\left(\{1,...,n\}\setminus J\right) \setminus  \{i\}$ and $\{i,j\}$.
	
	Suppose we are given a link $lk(K,A)=\left(p_{ij}\right)$ which is reducible for the index sets $L$ and $M$. For $k\in K$, what does the link $lk\left(K\setminus\{k\},A\right)=\left(q_{ij}\right)$ look like?
	
	Since $lk(K,A)$ is reducible we have:
	$$0=p_{lm}=-\frac{q_{kl}{q_{km}}}{q_{kk}}+q_{lm}$$
	Once again using the almost negative nature of these matrices we can deduce $q_{lm}=0$ for all $l\in L$ and $m\in M$. Additionally, either $q_{kl}=0$ for all $l\in L$ or $q_{km}=0$ for all $m\in M$.
	
	If $q_{km}=0$ for all $m\in M$, then $lk\left(K\setminus \{k\},A\right)$ is reducible for the index sets $L\cup\{k\}$ and $M$.\\
	If on the other hand we have $q_{kl}=0$ for all $l\in L$, then $lk\left(K\setminus \{k\},A\right)$ is reducible for the index sets $L$ and $M\cup \{k\}$.
	
	So iterating the last argument we have proven the claim of the Lemma if some link matrix $lk(J,A)$ contains a $0$-column for $\#J\leq n-2$. To prove the Lemma we only need to show that this is true for $\#J=n-1$ as well.
	
	\underline{Step 3:} $\#J=n-1$. We show that either $A$ is parabolic or reducible.
	
	If $ lk(J,A)=0$ then we consider the matrix $\left(J\setminus \{j\},A\right)=\begin{pmatrix}
		\beta_{jj} &\beta_{j2}\\
		\beta_{j2} &\beta_{22}
	\end{pmatrix}$, where $\beta_{jj}>0$, since the link with regard to $\{j\}$ is defined. Using the formula to compute links we can deduce
	$$0=-\frac{\beta_{j2}^2}{\beta_{jj}}+\beta_{22}\quad \Leftrightarrow \quad \beta_{jj}\beta_{22}=\beta_{j2}^2$$
	So we can conclude that $\beta_{22}\geq 0$. If $\beta_{22}=0$, then $\beta_{j2}=0$ and we are in the case we discussed first because $lk\left(J\setminus \{j\},A\right)$ contains a $0$-column.\\
	If $\beta_{22}>0$ we have
	$$lk\left(J\setminus \{j\},A\right)=\begin{pmatrix}
		\beta_{jj} & -\sqrt{\beta_{jj}\beta_{22}}\\
		-\sqrt{\beta_{jj}\beta_{22}} &\beta_{22}
	\end{pmatrix}.$$
	This matrix is a parabolic matrix, since $\beta_{jj}>0$ and $\beta_{22}>0$ and the determinant $\det( lk\left(J\setminus \{j\},A\right))=0$. 

 We now show in general that given an almost negative matrix $C$ with $c_{11}>0$, if $lk(\{1\},C)$ is parabolic, then $C$ is parabolic or $C$ is reducible.

 So assume $lk(\{1\},C)$ is parabolic. Moreover we assume that both $C$ and $lk(\{1\},C)$ are irreducible, because else we are done by Step 2. Due to \cite[Cor. 9.3]{Moussong}, we know that $C$ is parabolic if its smallest Eigenvalue is $0$ (the same holds for $lk(\{1\},C)$). We consider the vector $w=(w_1,...,w_m)$ and calculate $\langle w,w\rangle_C$. Let $\varphi_{e_1\perp}\colon \mathbb{R}^m\to \mathbb{R}^m$ denote the orthogonal projection onto $e_1^\perp$ with regard to $C$ (this is well-defined since $c_{11}>0$). We can write $w=ke_1+\sum_{j=2}^{m}\alpha_j\varphi_{e_1^\perp} (e_j)$ and since $\varphi_{e_1^\perp}$ is linear, $\alpha_j=w_j$ for $j\in\{2,...,m\}$. Moreover let $\pi_j\colon\mathbb{R}^m\to \mathbb{R}$ denote the canonical projection on the $j$-th coordinate, that is $\pi_j(x_1,...,x_m)=x_j$. We define $\bar{w}:=\left(\pi_2\left(\varphi_{e_1^\perp}(w)\right),...,\pi_m\left(\varphi_{e_1^\perp}(w)\right)\right)$ and obtain:
 \begin{equation*}
     \begin{split}
         \langle w,w\rangle_C
         &=\left\langle ke_1+\sum_{j=2}^{m}w_j\varphi_{e_1^\perp} (e_j), ke_1+\sum_{j=2}^{m}w_j\varphi_{e_1^\perp} (e_j)\right\rangle_C\\
         &= k^2c_{11}+\sum_{j,l=2}^{m} w_jw_l\langle \varphi_{e_1^\perp} (e_j),\varphi_{e_1^\perp} (e_l)\rangle_C\\
         &= k^2c_{11}  + \left\langle \bar{w},\bar{w}\right\rangle_{lk(\{1\},C)}
     \end{split}
 \end{equation*}
Since $c_{11}>0$ and $lk(\{1\},C)$ is parabolic (hence positive semi-definite), both terms in the last equation are non-negative. Hence $\langle w,w\rangle_C\geq 0$ for any vector $w\in \mathbb{R}^m$. Therefore all eigenvalues of $C$ are $\geq 0$ and hence $C$ is either positive definite or parabolic. But we know $C$ can't be positive definite since $lk(\{1\},C)$ is not (this can be seen e.g. geometrically using Proposition \ref{Proposition} (iii)), hence we conclude $C$ is parabolic.

This finishes Step 3 and the proof by applying this argument $\#J$ times to $(0)=lk(J,A)$.

\end{proof}

With this, we can finally close the gap in the proof of \cite[Thm. B]{Moussong}. This is done by proving \cite[Lemma 10.3]{Moussong} in the cases that currently have a gap. For better readability we include the statement here.
\begin{lemma}
    Let $A$ denote an almost negative matrix. If $g(N(A))=2\pi$, then $A$ either has a principal parabolic submatrix of order $\geq 3$ or a reducible submatrix $B=A_1\oplus A_2$ where $A_1$ and $A_2$ are not positive definite.
\end{lemma}
Before proving the lemma, we want to briefly recall the definition of the girth. The \textit{girth} is the infimum of lengths of closed geodesics or infinity if such geodesics don't exist.
\begin{proof}
    Note that by \cite[Prop. 10.1]{Moussong}, $g(N(A))\geq 2\pi$.

    If $A$ is such that \cite[Lemma 9.11]{Moussong} holds, then the original proof works fine. Thus we can assume that $A$ is an almost negative matrix such that there exists a link matrix $lk(I,A)$ with either a principal submatrix of order $\geq 2$ which has a row or column consisting entirely of zeroes or a $2\times2$ principal parabolic submatrix. Moreover we can assume that $A$ is normalized since multiplication with a diagonal matrix does not change the nerve complex.

     We argue via induction on the order $n$ of $A$. Clearly, if $n=1,2,3$, there is nothing to show. So let $n\geq 4$ and assume $q\colon \mathbb{S}^1\to N(A)$ is a closed geodesic in $N(A)$ of length $2\pi$ (this exists by \cite[Cor. 5.3]{Moussong}). We denote by $Q$ the image of $q$.

    \underline{Case 1:} There exists an $I\subseteq \{1,...,n\}$ such that $lk(I,A)$ contains a principal parabolic $2\times 2$ matrix. 
    
    By Lemma \ref{split} that means either $A$ contains a principal parabolic submatrix of order $\geq 3$ or a principal submatrix of the form $B=\begin{pmatrix}
        1 & 0&0\\
        0&1&-1\\
        0&-1&1
    \end{pmatrix},$ possibly after a permutation of the index set. In the former case we are done, so assume the latter.

    Let $v_1,v_2,v_3$ denote the vertices in $N(A)$ corresponding to $B$. We note that \begin{itemize}
    
        \item The union of all vertex sets of the supports of $q(t)$ is the vertex set of $N(A)$, else we can apply the induction hypothesis to a smaller matrix.
        \item There is a $t\in \mathbb{S}^1$ such that $q(t)\in St(v_1,N(A))-Ost(v_1,N(A))$, because every simplex is geodesically convex and any restriction of $q$ to an arc of length $<\pi$ is a geodesic.
    \end{itemize}
    Hence there exist two points $x,y\in St(v_1,N(A))- Ost(v_1,N(A))$ which lie on $Q$, that is there exist $t_1\neq t_2\in \mathbb{S}^1$ such that $x=q(t_1)$ and $y=q(t_2)$. Thus by \cite[Lemma 9.8]{Moussong} $Q\cap Clk(v_1,N(A))$ contains a portion of length $\pi$ of $Q$. 

    We now consider two subcases.\\
    
    \underline{Subcase 1:} $Q\subseteq St(v_1,N(A))$.\\
    If $Q\subseteq lk(v_1,N(A))$ we can apply the induction hypothesis and Lemma \ref{split} to finish the proof. So suppose that is not the case. Then by \cite[Lemma 8.2]{Moussong}, $St(v_1,N(A))\subseteq Slk(v_1,N(A))$. Moreover due to \cite[Cor. 9.9, 9.10]{Moussong}, $q$ is a closed geodesic in $Slk(v_1,N(A))$. Again by the induction hypothesis and Lemma \ref{split}, there are no closed geodesics of length $2\pi$ in $lk(v_1,N(A))$, so by \cite[Cor. 5.6]{Moussong}, $Q$ passes through both suspension points in $Slk(v_1,N(A))$, which is impossible since one of these points is not contained in $St(v_1,N(A))$. This proves subcase 1.

    \underline{Subcase 2:}  $Q\not\subseteq St(v_1,N(A))$.\\
    Consider the vertex sets of the supports of $x$ and $y$ respectively. If any edge length between $v_1$ and a vertex of such a support is $>\frac{\pi}{2}$, then the length of the segment of $Q$ connecting $x$ and $y$ in $St(v_1,N(A))$ is larger than $\pi$. Since $Q$ is not contained in $St(v_1,N(B))$, there exists a $t'\in \mathbb{S}^1$ and a vertex $v'\notin\{v_1,v_2,v_3\}$ such that $q(t')\in Ost(v',N(A))$. We can repeat the above argument to find $x',y'\in St(v',N(A))-Ost(v',N(A))$ such that the geodesic segment connecting $x'$ and $y'$ in $St(v',N(A))$ has length $\geq \pi$. Moreover, the length of $q$ is greater or equal to the sum of these two segments. Since $l(q)=2\pi$ we can deduce that both arcs have length $\pi$ and $x'\in \{x,y\}$ and $y'\in \{x,y\}$. Hence there are edges of length $\frac{\pi}{2}$ between all vertices in the union of the supports of $x$ and $y$ and $v_1$ and $v'$ respectively. Moreover there is no edge between $v_1$ and $v'$, since else $v_1,v'$ and the vertices of the supports of $x$ and $y$ would span a simplex (since all entires in the matrix are $0$, except the one corresponding to the edge between $v_1$ and $v'$, which is greater than $-1$). And a simplex does not contain a closed geodesic. Moreover, by the first bullet point, $v_2,v_3\in V(Supp(x))\cup V(Supp(y))$. Hence, by construction, the principal submatrices corresponding to the full subcomplexes spanned by $\{v_1,v'\}$ and by $\{v_2,v_3\}$ are not positive definite and ``commute'', meaning the principal submatrix of $A$ corresponding to their indices is reducible. This is what we wanted to show.

    \underline{Case 2:} There exists an $I\subseteq \{1,...,n\}$ such that $lk(I,A)$ contains a principal submatrix of order $\geq 2$ which has a row or column consisting entirely of zeroes. 
    
    Thus, $lk(I,A)$ contains a principal submatrix of the form $\begin{pmatrix}
        0 &0\\
        0& \alpha
    \end{pmatrix}$ for some $\alpha >0$ or a principal submatrix of the form $\begin{pmatrix}
        0 &0 \\
        0& 0
    \end{pmatrix}$ (all after a possible permutation of indices). In the second case, the calculations of the previous Lemma show that at least one diagonal entry of $A$ is zero. In that case we can remove the corresponding row and column to obtain a matrix $A'$ of order $n-1$ with $N(A)=N(A')$. Thus by induction hypothesis, we are done. 
    
    In the other case we deduce that $A$ contains a principal submatrix of the form $\begin{pmatrix}
        1 & -1 &0\\
        -1&1&0\\
        0&0&1
    \end{pmatrix}$ which we dealt with in the first case, which finishes the proof.
\end{proof}
The proof of the characterization of hyperbolic Coxeter groups provided by Moussong relies only on \cite[Cor. 10.2, 10.4]{Moussong}. The proof of the former of these is correct and the proof of the latter relies only on \cite[Lemma 10.3]{Moussong}, for which we have provided a fix. Therefore we have filled the gap in the characterization of hyperbolic Coxeter groups, which we include here for completeness.
\begin{theorem}(\cite[Thm 17.1]{Moussong}, \cite[Thm. 4.2.1]{P})
	Let $(W,S)$ denote a Coxeter system. The following are equivalent:
	\begin{enumerate}
		\item $W$ is Gromov-hyperbolic.
		\item $W$ has no subgroup isomorphic to $\mathbb{Z}\oplus \mathbb{Z}$.
		\item There is no subset $T\subseteq S$ such that $\left(W_T,T\right)$ is an affine Coxeter system of $rk\geq 3$ and \\
		there is no pair of disjoint subsets $T_1,T_2\subseteq S$ such that $W_{T_1}$ and $W_{T_2}$ commute and are infinite.
	\end{enumerate}
\end{theorem}

\end{document}